\documentclass[12pt,leqno]{article}
\usepackage{color}

\usepackage[T1]{fontenc}
\usepackage{fullpage}
\usepackage{graphicx}
\usepackage[utf8]{inputenc}
\usepackage[british]{babel}
\usepackage{setspace}
\usepackage{amsmath}
\usepackage{mathrsfs}
\usepackage{amssymb}
\usepackage{amsthm}
\usepackage{booktabs}
\usepackage{caption}
\usepackage{tikz-cd}
\usepackage{float}
\usepackage{rotating}
\usepackage{verbatim}
\usepackage{epigraph}
\usepackage{geometry}
\usepackage{datetime}
\usepackage{youngtab}
\usepackage{multirow}
\usepackage{hyperref}
\usepackage[normalem]{ulem}
\usepackage{makecell}
\usepackage{tabularx}

\def\id{{\rm id}}

\def\Oc{{\mathcal O}}

\usepackage{tikz}
\usepgflibrary{shapes.geometric}
\usepackage{amsmath,amsthm, amscd, amssymb, amsfonts, mathrsfs}
\usepackage[all]{xy}
\usepackage{color}
\usepackage{graphics}
\usepackage{amssymb}
\usepackage{amscd}
\usepackage{amsmath}
\input xy 
\xyoption{all}
\title{A parametrization of sheets of conjugacy classes in bad characteristic}
\newtheorem{theorem}{Theorem}[section]

\newtheorem{corollary}[theorem]{Corollary}
\newtheorem{proposition}[theorem]{Proposition}

\newtheorem{remark}[theorem]{Remark}

\author{Filippo Ambrosio, Giovanna Carnovale, Francesco Esposito\\
Dipartimento di Matematica ``Tullio Levi-Civita''\\
Torre Archimede - via Trieste 63 - 35121 Padova - Italy\\
email: ambrosio@math.unipd.it,  \\
carnoval@math.unipd.it, esposito@math.unipd.it}
\date{}

\begin{document}
\maketitle
\begin{abstract}
Let $G$ be a simple algebraic group of adjoint type over an algebraically closed field $k$ of bad characteristic. We show that its sheets of conjugacy classes are parametrized by $G$-conjugacy classes of pairs $(M,\Oc)$ where $M$ is the identity component of the centralizer of a semisimple element in $G$ and $\Oc$ is a rigid unipotent conjugacy class in $M$, in analogy with  the good characteristic  case.
We explicitly describe the possible choices for $M$.
\end{abstract}

\section{Introduction}

Sheets in a reductive algebraic group $G$ are the irreducible components of the locally closed subsets of $G$ consisting of conjugacy classes of the same dimension. They occur also as irreducible components of the strata the partition of $G$, defined in \cite{lustrata} in terms of Springer representations with trivial local system, \cite{gio-MR, gio-arc}. One of the most fascinating features of strata is that they are  parametrized by a family of irreducible representations of the Weyl group which depends on the root system of $G$ and not on the characteristic. It is therefore of interest to figure out the behaviour of the irreducible components of strata when the characteristic of the base field varies. 

A description of sheets in good characteristic, and a parametrization of sheets in terms of $G$-conjugacy classes of triples $(M, Z(M)^\circ s, \Oc)$ where $M$ is the identity component of the centralizer of a semisimple element in $G$, $Z(M)^\circ s$ is a suitable coset in the component group $Z(M)/Z(M)^\circ$, and $\Oc$ is a rigid unipotent conjugacy class in $M$ was given in \cite{gio-espo} in good characteristic, and extended to the case of bad characteristic in \cite{simion}.  A refinement of this parametrization in terms of pairs $(M,\Oc)$ where $M$ and $\Oc$ are as above was given in \cite{gio-arc} under the assumption that $G$ is simple of adjoint type and the characteristic of the base field is good for $G$. 
The present paper  answers a question by G. Lusztig on the extension to arbitrary characteristic of this parametrization of sheets. 

Observe that, even if the formulation of the statement is the same, the collection of possible  centralizers of a semisimple element in $G$ varies with the characteristic of the base field, as well as the collection of unipotent conjugacy classes. Centralizers of semisimple elements are fewer in bad characteristic than in good characteristic, \cite{carter, carter-classical,der} whilst the number of unipotent conjugacy classes may increase when passing from good to bad  characteristic.

We therefore elaborate upon  results in \cite{carter, carter-classical,der} in the spirit of \cite{mcninch-sommers} in order to provide a combinatorial description of the root systems of centralizers of semisimple elements, which will allow us to retrieve most ingredients that were necessary for the proof of \cite[Theorem 4.1]{gio-arc}. This will allow us to show that the number of sheets in type $G_2$ is the same  in good  characteristic and in characteristic $3$, but it is smaller in characteristic $2$. 
In \cite{lustrata} G. Lusztig defined sets of representations of the Weyl group ${\mathcal S}_2(W)$ and ${\mathcal S}_2^p(W)$ for every prime $p$.  We wonder whether for those primes $p$ for which ${\mathcal S}_2(W)={\mathcal S}_2^p(W)$, the number of sheets for $G$ simple of adjoint type in good characteristic  and in characteristic $p$ are equal. 

\section{Notation}Let $G$ be a connected reductive algebraic group defined over an algebraically closed field $k$ of characteristic exponent $p$. Let $\Phi$ be the root system of $G$ and $\Delta=\{\alpha_1,\,\ldots,\,\alpha_n\}$ be a fixed base of $\Phi$. If $\Phi$ is irreducible,  the numbering of simple roots will be as in \cite{bourbaki} and we will denote by $\alpha_0$ the opposite of the highest root in $\Phi$ and by $d_i$  the coefficient of $\alpha_i$ in the expression of $-\alpha_0$. We set $d_0:=1$ and $\widetilde{\Delta}=\Delta\cup\{\alpha_0\}$. 
For a subset $S\subset \widetilde{\Delta}$ we define $d_S:={\rm gcd}(d_i~|~\alpha_i\not\in S)$. In particular, $d_S=1$ if $S\subset \Delta$. 

The group acts on itself by conjugation and we denote $g\cdot h=ghg^{-1}$ for $g,\,h\in G$ and $G\cdot g$ the $G$-conjugacy class of $g\in G$. For a closed subgroup $H\leq G$,   the identity component will be denoted by $H^\circ$, and for $g\in G$ the centralizer will be denoted by $G_g$. We will call $G_g^\circ$ the connected centralizer of $g$. The Jordan decomposition of an element $g\in G$ will be usually denoted by $g=su$.  
 
For $m\in{\mathbb N}$ we set $G_{(m)}:=\{g\in G~|~\dim (G\cdot g)=m\}$.  These sets are locally-closed and their irreducible components are called the {\em sheets} of the $G$-action. 
For  $Z\subset G$ we also define $m_Z:=\max\{m\in{\mathbb N}~|~G_{(m)}\cap Z\neq\emptyset\}$ and  $Z^{reg}:=Z\cap G_{(m_Z)}$, and  $C_G(Z)$ will indicate the centralizer of $Z$ in $G$. 

If we  fix a maximal  torus $T$ of $G$ and $\Phi$ is the root system of $G$ with respect to $T$, then for $\alpha\in \Phi$ we indicate by $X_\alpha$ the corresponding root subgroup. For a closed subset $\Psi\subset \Phi$, (see \cite[VI, n. 1.7, D\'efinition 4]{bourbaki}),  and for $s\in T$ we set 
$G_\Psi:=\langle T,\,X_\alpha~|~\alpha\in\Psi\rangle$.



\subsection{Construction of sheets and a first parametrization}It was observed in  \cite[\S3]{lusztig-inventiones} that  $G$ has a partition into finitely many, locally closed, smooth, irreducible, $G$-stable sets, which we call Jordan classes, each contained in some $G_{(m)}$. As a set, the  class containing $g=su$ is \begin{align*}J(su)=G\cdot((Z(G^{\circ}_s)^\circ s)^{reg}u).\end{align*} In other words, a $G$-conjugacy class  lies in $J(su)$ if and only if it contains an element with Jordan decomposition $s'u$ with $G^{\circ}_{s}=G^{\circ}_{s'}$ and $s'\in Z(G^{\circ}_s)^\circ s$.  The closure  of a Jordan class is a  union of Jordan classes, \cite[\S3]{lusztig-inventiones}, hence the same holds for the regular locus of the closure of a Jordan class. This gives a partial order on the set of Jordan classes given by $J_1\leq J_2$ if and only if $J_1\subset\overline{J_2}^{reg}$. The sheets in $G$ are the locally closed sets of the form $\overline{J}^{reg}$ where $J$ is maximal with respect to $\leq$. Hence the set of sheets in $G$ is in bijection  with  the set  $\mathcal{J}$ consisting of maximal Jordan classes \cite[Proposition 5.1]{gio-espo}, \cite[\S 3]{simion}.

If $J=J(su)$ as above, then \cite[Proposition 4.8]{gio-espo},\cite[Lemma 2.1]{gio-arc}  give:
\begin{align*}\label{eq:reg-closure-group}
\overline{J(su)}^{reg}&=\bigcup_{z\in Z(G_s^\circ)^\circ s}G\cdot(z{\rm Ind}_{G_s^\circ}^{G_z^\circ}(G_s^\circ\cdot u))
\end{align*}
where ${\rm Ind}_{G_s^\circ}^{G_z^\circ}(G_s^\circ\cdot u)$ is Lusztig-Spaltenstein's induced unipotent conjugacy class, \cite{lusp}. The maximal Jordan classes are precisely those for which the class of $u$ is {\em rigid} in $G_s^\circ$, i. e., it is not induced from any unipotent class in a proper Levi subgroup of a parabolic subgroup of $G_s^\circ$, \cite[Proposition 5.3]{gio-espo}, \cite[Lemma 2.4]{simion}. 

Sheets are then parametrized as follows.

\begin{theorem}{\rm (\cite[Theorem 5.6]{gio-espo}, \cite[Theorem 3.1]{simion})}\label{thm:parametrization} The assignment  $J=J(su)\mapsto (G_s^\circ, Z(G_s^\circ)^\circ s, G_s^\circ\cdot u)$ induces a bijection between ${\mathcal J}$ and the set of $G$-orbits of triples $(M,Z(M)^\circ r,\Oc)$ where $M$ is the connected centralizer of a semisimple element of $G$; $Z(M)^\circ r$ is a coset in $Z(M)/Z(M)^\circ$ satisfying $C_G(Z(M)^\circ r)^\circ=M$  and $\Oc$ is a rigid unipotent conjugacy class in $M$.\end{theorem}\hfill$\Box$

We aim at a simpler parametrization for $G$ simple and of adjoint type. 

\subsection{Connected centralizers of semisimple elements}

In this subsection $G$ is quasisimple.  Identity components of centralizers of semisimple elements have been studied in \cite{carter,carter-classical,der,mcninch-sommers}. In the spirit of the latter, we give a combinatorial characterization of the root subsystem of such subgroups when the base field is an arbitrary algebraically closed field. We believe it to be well-known but we could not locate a proper reference. 

\begin{proposition}\label{prop:centralizer}
Let $G$ be quasisimple,  $T$ be a maximal torus in $G$ and $\Psi$ be a closed subset of $\Phi$. Then $G_\Psi$ is the connected centralizer of an element in $T$ if and only if  
$\Psi$ is conjugate to a root subsystem $\Psi'$  admitting a base $\Delta_{\Psi'}\subset\widetilde{\Delta}$ and such that ${\rm gcd}(p,\,d_i~|~i\in\widetilde{\Delta}\setminus \Delta_{\Psi'})=1$ by an element in the normalizer $N(T)$ of $T$.
\end{proposition}
\begin{proof}If $\Psi$ is $N(T)$-conjugate to a root subsystem $\Psi'$  admitting a base $\Delta_{\Psi'}\subset\widetilde{\Delta}$ and such that ${\rm gcd}(p,\,d_i~|~i\in\widetilde{\Delta}\setminus \Delta_{\psi'})=1$, then there exists an $\alpha_{\overline{i}} \in\widetilde{\Delta}\setminus \Delta_{\Psi'}$ such that $p\not|d_{\overline{i}}$. Replacing $a_1$ with $d_{\overline{i}}$ in the proof of \cite[Proposition 32]{mcninch-sommers}, we obtain an element $s\in Z(G_\Psi)$ such that $G_s^\circ=G_\Psi$.

Assume now that $G_\Psi=G_s^\circ$ for some $s\in T$. By \cite[Proposition 30]{mcninch-sommers} the subgroup $G_{\Psi}$ is $G$-conjugate to some $G_{\Psi'}$ where $\Psi'$ admits a base $\Delta_{\Psi'}\subset\widetilde{\Delta}$. Conjugacy of maximal tori in $G_{\Psi'}$ ensures that conjugation of the two subgroups, and of the corresponding root systems can be obtained using an element in $N(T)$. We show that ${\rm gcd}(p, d_{\Delta_{\Psi'}})=1$. If the characteristic exponent $p=1$, then there is nothing to prove. Assume  for a contradiction that $p>1$ and divides  $d_i$ for every $\alpha_i\in \widetilde{\Delta}\setminus \Delta_{\Psi'}$. Since $d_0=1$, in this situation $\alpha_0\in\Delta_{\Psi'}$ and $G_{\Psi'}$ is never the Levi subgroup of a parabolic subgroup of $G$.\\
Let $\{\omega_j^\vee,\,j=1,\,\ldots,\,n\}$ be the basis for the cocharacters of $T$ dual to $\Delta$ and let $t\in T$ be such that $G_{\Psi'}=G_t^\circ$.  
Since $\alpha_i(t)=1$ for every $\alpha_i\in\Delta_{\Psi'}$, we have $t=\prod_{j=1,\ldots,\,n\atop 
\alpha_j\not\in\Delta_{\Psi}}\omega_j^\vee(\zeta_j)$, and $\alpha_0(t)=1$ gives 
\begin{equation}\label{eq:equality}
\prod_{j\geq 1,\, 
\alpha_j\not\in\Delta_{\Psi'}}\zeta_j^{d_j}=1,  \textrm{ so also } \prod_{j\geq 1,\,
\alpha_j\not\in\Delta_{\Psi'}}\zeta_j^{d_j/p^l}=1\textrm{ if }p^l|d_{\Delta_{\Psi'}}.
\end{equation}
By \cite[VI \S 1, n. 1.7, Prop 24]{bourbaki} the system $\Psi'$ is not ${\mathbb Q}$-closed, i.e., there exists an $\alpha\in\left({\mathbb Q}\Delta_{\Psi'}\cap \Phi\right)\setminus \Psi'$. 
In other words, there exist $a,\,b,\, a_i,\,b_i\in{\mathbb Z}$, for $i=1,\,\ldots,\,n$ such that $\alpha_i\not\in\Delta_{\Psi'}$, with $b, b_i\neq0$ and $z_j\in {\mathbb Z}$ for $j=1,\,\ldots,\,n$ such that  
\begin{align*}
\alpha=\frac{a}{b}\alpha_0+\sum_{\alpha_i\in\Delta_{\Psi'}\cap\Delta}\frac{a_i}{b_i}\alpha_i=\sum_{j=1}^nz_j\alpha_j\in\Phi\setminus\Psi'.
\end{align*}
Hence $b|d_i$ for all $\alpha_i\not\in \Delta_{\Psi'}$, that is, $b|d_{\Delta_{\Psi'}}$. 
Observe that $d_{\Delta_{\Psi'}}$ is either $p^l$ for some $l\geq 1$ or else $d_{\Delta_{\Psi'}}=6$, $\Phi=E_8$, $p=2$ or $3$, and $\Delta_{\Psi'}=\widetilde{\Delta}\setminus\{\alpha_4\}$.
In the first case \eqref{eq:equality} gives $\alpha(t)=1$, a contradiction. In the second case, $t=\omega^\vee_4(\zeta_4)$ and \eqref{eq:equality} becomes $\zeta_4^{6/p}=1$. Let $\beta\in\Phi\setminus\Psi'$ be such that its coefficient of $\alpha_4$ in its expression as a sum of simple roots is $6/p$. Then  again $\beta(t)=1$, a contradiction. \\
\end{proof}

\begin{proposition}\label{prop:mc-sommers}Let $G$ be simple of adjoint type,  $T$ be a maximal torus in $G$, and $\Psi\subset \Phi$ be a closed subsystem such that 
 $G_{\Psi}=G_s^\circ$ for some $s\in T$. Assume in addition that $\Psi$ admits a base $\Delta_{\Psi}\subset\widetilde{\Delta}$. Then
\begin{enumerate}
\item[{\rm (a)}] The torsion subgroup of ${\mathbb Z}\Phi/{\mathbb Z}\Psi$ is ${\mathbb Z}/d_{\Delta_\Psi}{\mathbb Z}$.
\item[{\rm (b)}] $Z(G_\Psi)/Z(G_\Psi)^\circ$ is cyclic of order $d_{\Delta_{\Psi}}$. 
\item[{\rm (c)}] For $t\in Z(G_\Psi)$ we have $C_G(Z(G_\Psi)^\circ t)=G_\Psi$ if and only if  $Z(G_\Psi)^\circ t$ is a generator of $Z(G_\Psi)/Z(G_\Psi)^\circ$.
\end{enumerate}
\end{proposition}
\begin{proof}
(a) This is observed in \cite[\S 2]{sommers}. \\ (b)  The order of the torsion subgroup of ${\mathbb Z}\Phi/{\mathbb Z}\Psi$ is coprime with $p$ by Proposition \ref{prop:centralizer}. Hence we are in a position to use the argument in  \cite[\S 2.1]{sommers}, \cite[Lemma 33]{mcninch-sommers}, that we sketch for completeness. By construction  and \cite[Proposition 3.8]{MT}  from which we borrow notation, $Z(G_\Psi)=({\mathbb Z}\Psi)^\perp$, $Z(G_\Psi)^\perp={\mathbb Z}\Psi$, and the character group $X(Z(G_\Psi))$ is ${\mathbb Z}\Phi/{\mathbb Z}\Psi$. Then  $$X(Z(G_\Psi)/Z(G_\Psi)^\circ)\simeq\{\chi\in X(Z(G_\Psi))~|~\chi(z)=1,\,\forall z\in Z(G_\Psi)^\circ\}$$ consists of those torsion elements in ${\mathbb Z}\Phi/{\mathbb Z}\Psi$ whose order is coprime with $p$, that is,  ${\mathbb Z}/d_{\Delta_\Psi}{\mathbb Z}$. \\
(c) In good characteristic this is \cite[Lemma 34 (1)]{mcninch-sommers}, whose proof relies on the fact that every character of $Z(G_\Psi)/Z(G_\Psi)^\circ$ can be represented by an element in $\Phi$. The proof of this statement depends on considerations on root systems which are characteristic-free and on a natural isomorphism between $X(Z(G_\Psi)/Z(G_\Psi)^\circ)$ and ${\mathbb Z}({\mathbb Q}\Psi\cap\Phi)/{\mathbb Z}\Psi$, which remains valid  because $X(Z(G_\Psi)/Z(G_\Psi)^\circ)$ is the torsion subgroup of ${\mathbb Z}\Phi/{\mathbb Z}\Psi$ also in our situation.\end{proof} 

\subsection{The parametrization}

In this Section $G$ is simple of adjoint type. We are now in a position to prove the refinement of the parametrization of sheets of $G$. The general case can be readily deduced by standard arguments.

\begin{theorem}\label{thm:refinement}
Let $G$ be simple and of adjoint type. The sheets in $G$ are in bijection with the $G$-conjugacy classes of pairs $(M,\Oc)$ where $M$ is the connected centralizer of a semisimple element in $G$ and $\Oc$ is a rigid unipotent conjugacy class in $M$. 
\end{theorem}
\begin{proof}
In good characteristic this is \cite[Theorem 4.1]{gio-MR}, so we assume that $p$ is bad for $G$. Sheets are parametrized by triples $(M, Z(M)^\circ s,\Oc)$ as in Theorem \ref{thm:parametrization}.
The assignment $(M, Z(M)^\circ s,\Oc)\mapsto (M,\,\Oc)$ induces a well-defined and surjective map between the set of $G$-conjugacy classes of  triples and the set of $G$-conjugacy classes of pairs as above. We show injectivity of this map.  If $G$ is classical, then $M$ is a Levi subgroup of a parabolic subgroup of $G$, for any pair $(M,\Oc)$, hence $Z(M)=Z(M)^\circ$ and there is nothing to prove, so we assume that $G$ is of exceptional type.  

Let  $(M,Z(M)^\circ s,\Oc)$ and $(M,Z(M)^\circ r,\Oc)$ be two triples inducing the same image. Without loss of generality $s\in T$,  $M=G_\Psi$ where $\Psi$ has base  $\Delta_\Psi\subset \widetilde{\Delta}$, and  $Z(M)^\circ s,\,Z(M)^\circ r\subset T$.  If $d_{\Delta_{\Psi}}\leq 2$, then necessarily $Z(M)^\circ r= Z(M)^\circ s$, so we assume that  $d_{\Delta_{\Psi}}\geq 3$.
 
By \cite[Proposition 7]{sommers}  there is a $w$ in the stabilizer $N_W(\Delta_\Psi)$ of $\Delta_{\Psi}$ in $W$ whose action on ${\mathbb Z}\Phi/{\mathbb Z}\Psi$ generates the automorphism group of the torsion subgroup of ${\mathbb Z}\Phi/{\mathbb Z}\Psi$, which is isomorphic to $Z(M)/Z(M)^\circ$ by Propoisition \ref{prop:mc-sommers} (a).  We claim that any representative of $w$ in $N(T)$ preserves $\Oc$.  Since  rigid unipotent classes in type $A$ are trivial,  it is enough to consider only the case in which $\Delta_\Psi$ contains a (necessarily unique) component of type different from $A$, and $\Oc$ is non-trivial in the corresponding subgroup. Such a component is always preserved by the action of $w$.  The list of $\Delta_{\Psi}$ with $d_{\Delta_{\Psi}}\geq 3$ in the proof of \cite[Proposition 7]{sommers} shows  that we only need to consider two cases for $G$ of type $E_8$, namely $\Delta_{\Psi}=A_3+ D_5$ which may occur only when $p=3,5$,  and $\Delta_{\Psi}=A_2+ E_6$. Unipotent conjugacy classes in type $D$ for  $p=3$ and $5$ are characteristic unless their Jordan form corresponds to a  very even partition, i.e., a  partition with only even terms, each occurring an even number of times. Such partitions never occur in $D_n$ for $n$ odd.
Rigid unipotent conjugacy classes in $E_6$ in arbitrary characteristic can be deduced from \cite[Chapitre II. Appendice]{spalt} and they are characteristic for dimensional reasons. 
\end{proof}

Fixing a maximal torus $T$ in $G$, by standard arguments we retrieve a parametrization of sheets by orbits of the Weyl group. 
We set ${\mathcal T}$ to be the set of pairs $(M,\Oc)$ where $M$ is the connected centralizer of an element in $T$ and $\Oc$ is a rigid unipotent class in $M$, so $N(T)$ and $W=N(T)/T$ naturally act on $\mathcal{T}$.  

\begin{corollary} 
Let $G$ be simple and of adjoint type. The sheets in $G$ are parametrized by elements in ${\mathcal T}/W$. 
\end{corollary}\hfill$\Box$

%

\subsection{On the number of sheets in $G$}

It was observed in  \cite[Remark 3.3]{simion} that for $G$ of type $B_2$, the number of sheets  is independent of the characteristic and it was suggested this to hold in general for $G$ connected and simply connected.  This fails in general  because there exist sheets that are obtained from one another by multiplication by a central element, and such central element might no longer exist in bad characteristic: for example in  $G=SL_2(k)$, the sheets are: $\{\id\}$, $\{-\id\}$ and $G^{reg}$ for $p\neq 2$ and  $\{\id\}$ and $G^{reg}$ for $p=2$. The following remark shows that the number of sheets depends on $p$ also for $G$ simple of adjoint type. 
\begin{remark}
{\rm Let $\Phi=G_2$.  We use the parametrization in  Theorem \ref{thm:refinement}. The semisimple parts of the connected centralizers for $p$ good are of type:  $G_2$, $A_2$, $A_1+\widetilde{A}_1$, $A_1$, $\widetilde{A}_1$
 or conjugate to $T$. In type $A$ all rigid unipotent classes are trivial, and there are $3$ rigid unipotent classes in $G$, \cite[Chapitre II. Appendice]{spalt}. Hence, there are $8$ sheets for $p$ good.

According to Proposition \ref{prop:lista} the semisimple parts of the connected centralizers for $p=2$ are of type:  $G_2$, $A_2$, $A_1$, $\widetilde{A}_1$
 or conjugate to $T$, and there are $3$ rigid unipotent classes in $G$, \cite[Chapitre II. Appendice]{spalt}. Hence, there are $7$ sheets for $p=2$.
 
According to Proposition \ref{prop:lista} the semisimple parts of the connected centralizers for $p=3$ are of type:  $G_2$, $A_1+\widetilde{A}_1$,  $A_1$, $\widetilde{A}_1$ or conjugate to $T$, and there is an extra unipotent conjugacy class in $G_2$ which is rigid, as it can be deduced from the list of induced classes in \cite[Chapitre II. Appendice]{spalt}. Hence, the number of sheets for $p=3$ equals the number of sheets for $p$ good. }
\end{remark}


\section*{Appendix}

For the reader's convenience and for further reference we list the possible connected centralizers in bad characteristic, obtained making use of the analysis of $W$-conjugacy  classes of subsets of $\widetilde{\Delta}$ in \cite[\S 2.2]{sommers}. In most cases these classes are determined by their isomorphism type and the root lengths.  For $\Phi=E_7$ and $E_8$ we remove ambiguities adopting, as in {\it loc. cit.}, Dynkin's convention. Namely, for $n= 7,8$ we decorate with one prime the root subsystems which can be embedded in the subsystem of type $A_n$ within $E_n$ while we decorate with two primes the root subsystems with the same label which cannot be embedded in $A_n \subset E_n$. 

\begin{proposition}\label{prop:lista}
Let $G$  be quasisimple with $p$ bad for $G$. Let $T$ be a maximal torus in $G$ and $\Psi\subset \Phi$ be a closed subsystem with base $\Delta_\Psi\subset\widetilde{\Delta}$. 

If $\Phi$ is of classical type, then $G_\Psi$ is the connected centralizer of an element in $T$ if and only if it is the Levi subgroup of a parabolic subgroup of $G$, \cite{carter-classical}.

If $\Phi$ is of  exceptional type, then $G_\Psi$ is  the connected centralizer of an element in $T$  unless $G,\,p$ and $\Psi$ occur in Table \ref{tab:table}.
\end{proposition}

\begin{table}[H]
\centering
{\renewcommand{\arraystretch}{1.3}
\begin{tabularx}{\textwidth}{|l|l|X|}
\hline
$G$                       & $p$                & $\Psi$\\ \hline
$ E_6 $ &
$2$
& 
 $4A_1$, $A_3 + 2A_1$, $A_5 + A_1$ \\ \hline
$ E_6 $ &
$3$
&
 $3A_2$ \\ \hline
{$ E_7 $} &
{$2$} & 
 $(4A_1)'$, $(A_3 + 2A_1)'$, $(A_5+ A_1)'$, $A_7$, $5A_1$, $A_3 + 3A_1$, $2A_3$, $D_4 + 2A_1$, $D_6 + A_1$, $2A_3 + A_1$ \\ \hline
$ E_7 $ &
$3$
& 
 $3A_2$, $3A_2 + A_5$ \\ \hline
$ E_8 $ &
$2$& 
$(4A_1)''$, $(A_3 + 2A_1)''$, $(2A_3)''$, $(A_1 + A_5)''$, $A''_7$, $5A_1$, $3A_1 + A_3$, $2A_1+ D_4$, $4A_1 + A_2$, $2A_1 + A_5$, $A_3 + D_4$, $A_1 + D_6$, $A_1 + 2A_3$, $ 2A_1 + A_2 + A_3$, $2A_1 + D_5$, $D_8$,
$A_1 + A_7$, $A_1 + A_2 + A_5$, $A_3 + D_5$, $A_1 + E_7$ \\ \hline
$ E_8 $ &
$3$
&
 $3A_2$, $A_2 + A_5$, $3A_2 + A_1$, $ A_8$, $A_1 + A_ 2 + A_5$, $A_2 + E_6$ \\ \hline
$ E_8 $ &
$5$
& 

 $2A_4$ \\ \hline
$ F_4 $ &
$2$
&
 $2 A_1$, $ 2A_1 + \tilde A_1$, $A_1 + B_2$, $A_3$, $A_1 + C_3$, $A_3 + \tilde A_1$,
$B_4$ \\ \hline
$ F_4 $ &
$3$
& 
 $A_2 + \tilde A_2$ \\ \hline
$ G_2 $ &
$2$
& 
 $A_1+ \tilde A_1$ \\ \hline
$ G_2 $ &
$3$
& 
 $A_2$ \\ \hline
  \end{tabularx}}
  \caption{\label{tab:table}Root subsystems to be discarded }
\end{table}

\section{Acknowledgements}
We wish to  thank Prof. G. Lusztig for asking the question leading to this short note.  The authors acknowledge support by:  DOR2207212/22 ``Algebre di Nichols, algebre di Hopf e gruppi algebrici" and  BIRD203834 ``Grassmannians, flag varieties and their generalizations." funded by the University of Padova. They are members of the  INdAM group GNSAGA.

\end{document}